\documentclass[a4]{amsart}

\input xypic
\usepackage{amssymb}

\newtheorem{theorem}{Theorem}[section]
\newtheorem{proposition}[theorem]{Proposition}
\newtheorem{lemma}[theorem]{Lemma}
\newtheorem{corollary}[theorem]{Corollary}

\theoremstyle{definition}

\newtheorem{remark}[theorem]{Remark}

\newcommand{\g}{\ensuremath{\mathcal{G}}} 
\newcommand{\gk}{\ensuremath{\mathcal{G}_{k}}}

\newcommand{\cptwo}{\ensuremath{\mathbb{C}P^{2}}} 
\newcommand{\map}{\ensuremath{\mbox{Map}}} 
\newcommand{\mapstar}{\ensuremath{\mbox{Map}^{\ast}}} 
\newcommand{\mapcptwo}{\ensuremath{\mapstar(\cptwo,BSp(2))}} 

\newcommand{\hlgy}[1]{\ensuremath{H_{*}(#1)}}

\newcounter{bean}
\newenvironment{letterlist}{\begin{list}{\rm ({\alph{bean}})}
      {\usecounter{bean}\setlength{\rightmargin}{\leftmargin}}}
      {\end{list}}

\newcommand{\namedright}[3]{\ensuremath{#1\stackrel{#2}
 {\longrightarrow}#3}}
\newcommand{\nameddright}[5]{\ensuremath{#1\stackrel{#2}
 {\longrightarrow}#3\stackrel{#4}{\longrightarrow}#5}}
\newcommand{\namedddright}[7]{\ensuremath{#1\stackrel{#2}
 {\longrightarrow}#3\stackrel{#4}{\longrightarrow}#5
  \stackrel{#6}{\longrightarrow}#7}} 
\newcommand{\nameddddright}[9]{\ensuremath{#1\stackrel{#2}
 {\longrightarrow}#3\stackrel{#4}{\longrightarrow}#5
  \stackrel{#6}{\longrightarrow}#7\stackrel{#8}{\longrightarrow}#9}}

\newcommand{\larrow}{\relbar\!\!\relbar\!\!\rightarrow}

\newcommand{\lnamedright}[3]{\ensuremath{#1\stackrel{#2}
 {\larrow}#3}}
\newcommand{\lnameddright}[5]{\ensuremath{#1\stackrel{#2}
 {\larrow}#3\stackrel{#4}{\larrow}#5}}
\newcommand{\lnamedddright}[7]{\ensuremath{#1\stackrel{#2}
 {\larrow}#3\stackrel{#4}{\larrow}#5
  \stackrel{#6}{\larrow}#7}}

\newcommand{\qqed}{\hfill\Box}

\newcommand{\zmodtwo}{\ensuremath{\mathbb{Z}/2\mathbb{Z}}}

\newcommand{\zmodfour}{\ensuremath{\mathbb{Z}/4\mathbb{Z}}} 
\newcommand{\zmodeight}{\ensuremath{\mathbb{Z}/8\mathbb{Z}}}

\begin{document}


\title[$Sp(2)$-gauge groups]{The homotopy types of $Sp(2)$-Gauge groups over closed, 
         simply-connected four-manifolds} 
\author{Tseleung So}
\address{Mathematical Sciences, University
         of Southampton, Southampton SO17 1BJ, United Kingdom}
\email{tls1g14@soton.ac.uk}
\author{Stephen Theriault}
\address{Mathematical Sciences, University
         of Southampton, Southampton SO17 1BJ, United Kingdom}
\email{S.D.Theriault@soton.ac.uk}



\begin{abstract} 
We determine the number of distinct fibre homotopy types for the gauge groups of 
principal $Sp(2)$-bundles over a closed, simply-connected four-manifold. 
\end{abstract} 

\maketitle

\section{Introduction}
\label{sec:intro} 

Let $X$ be a pointed topological space, $G$ a topological group and 
\(\namedright{P}{}{X}\) 
a principal $G$-bundle. The \emph{gauge group} $\g(P)$ of $P$ is the 
group of $G$-equivariant automorphisms of $P$ that fix~$X$. The topology 
of gauge groups is of interest due to its connections with various moduli 
spaces~\cite{AB,BGG} and Donaldson Theory~\cite{D}. 

Considerable effort has gone into determining the homotopy types of 
gauge groups for specific groups $G$ and spaces $X$. Typically, $G$ 
and $X$ are chosen because of their interest in geometry or physics. 
In general, Crabb and Sutherland~\cite{CS} showed that, even if there are 
infinitely many inequivalent classes of principal $G$-bundles $P$ over $X$, 
there are only finitely many distinct homotopy types for their gauge groups. 
Precise enumerations of the homotopy types have been made 
in the following cases: $SU(2)$-bundles over $S^{4}$~\cite{K} or 
over a closed, simply-connected four-manifold~\cite{KT}; $SU(3)$-bundles 
over $S^{4}$~\cite{HK} or over a closed, simply-connected four-manifold~\cite{Th3}; 
$SU(5)$-bundles over~$S^{4}$~\cite{Th4}; $SO(3)$-bundles over 
$S^{4}$~\cite{KKKT}; and $Sp(2)$-bundles over $S^{4}$~\cite{Th2}. 
Substantial information has also been obtained for $SU(4)$-bundles 
over $S^{4}$~\cite{CT}; $Sp(3)$-bundles over $S^{4}$~\cite{C}; and 
$G_{2}$-bundles over $S^{4}$~\cite{KTT}. 

In this paper we consider the homotopy types of the gauge groups of 
principal $Sp(2)$-bundles over a closed, simply-connected four-manifold. 
It is well known that the principal $Sp(2)$-bundles over a closed 
simply-connected four-manifold $M$ are classified by their second Chern class, 
which can take any integer value. Let 
\(\namedright{P_{k}}{}{M}\) 
be the principal $Sp(2)$-bundle classified by the integer $k$, and 
let~$\gk(M)$ be its gauge group. We say that $\gk(M)$ is fibre homotopy 
equivalent to $\g_{\ell}(M)$ if there is a homotopy commutative diagram 
\[\spreaddiagramcolumns{-1pc}
  \diagram   
       \gk(M)\rrto^-{\simeq}\drto & & \g_{\ell}(M)\dlto \\ 
       & G. &
  \enddiagram\] 
For integers $a$ and $b$ let $(a,b)$ be their 
greatest common denominator. 

\begin{theorem} 
   \label{count} 
   Let $M$ be a closed, simply-connected four-manifold. The following hold: 
   \begin{letterlist} 
      \item if $(40,k)=(40,\ell)$ then $\gk(M)$ is fibre homotopy equivalent to $\mathcal{G}_{\ell}(M)$ 
                when localized rationally or at any prime;  
      \item if $\gk(M)$ is fibre homotopy equivalent to $\mathcal{G}_{\ell}(M)$ then $(40,k)=(40,\ell)$.  
   \end{letterlist} 
\end{theorem} 

Two comments should be made. First, existing statements enumerating the 
homotopy types of gauge groups are phrased in terms of homotopy equivalence 
rather than fibre homotopy equivalence. A fibre homotopy equivalence between 
$\gk(M)$ and $\g_{\ell}(M)$ is stronger than a homotopy equivalence, so part~(a) 
implies a corresponding homotopy equivalence statement. In part~(b) the 
stronger condition of fibre homotopy equivalence is used. In Theorem~\ref{Acount} 
we show that if $\gk(M)$ is homotopy equivalent to $\mathcal{G}_{\ell}(M)$ then 
$(20,k)=(20,\ell)$. However, it is not clear if this can be improved to $(40,k)=(40,\ell)$. 
Second, an interesting feature of the 
proof of Theorem~\ref{count}~(b) involves showing that the boundary map 
in a fibration involving the classifying space $B\gk(M)$ has ``order" $40$ 
(what is meant by ``order" is described precisely following Proposition~\ref{cptwomethod}). 
This is interesting because the analogous boundary map in the case of gauge 
groups for principal $SU(2)$ or $SU(3)$-bundles over~$M$ has its ``order"  
vary depending on whether $M$ is Spin or non-Spin. In the $Sp(2)$-case 
the ``order"  is independent of the existence of a Spin structure. One 
wonders if this holds more generally.

\section{Framing the problem} 
\label{sec:frame} 

This section summarizes what is known and what remains to be proved 
in Theorem~\ref{count}. In general, let $G$ be a simply-connected, simple 
compact Lie group and let $M$ be a closed, simply-connected four-manifold. 
Since $[M,BG]\cong\mathbb{Z}$ the principal $G$-bundles over $M$ are 
classified by the second Chern class of $M$. Let 
\(\namedright{P_{k}}{}{M}\) 
be a principal $G$-bundle whose second Chern class is $k\in\mathbb{Z}$. 
Let $\gk(M)$ be the gauge group of $P_{k}$. 

The following decomposition was proved for the Spin case in~\cite{Th1} and for the 
non-Spin case in~\cite{So}. 

\begin{theorem} 
   \label{Mgaugedecomp} 
   Let $G$ be a simply-connected, simple compact Lie group and let $M$ 
   be a closed, simply-connected four-manifold with $H^{2}(M;\mathbb{Z})$ 
   of rank $d\geq 1$.  
   \begin{letterlist} 
      \item If $M$ is Spin then there is an integral homotopy equivalence 
                \[\gk(M)\simeq\gk(S^{4})\times\prod_{i=1}^{d}\Omega^{2} G.\] 
      \item If $M$ is non-Spin then there is an integral homotopy equivalence 
                \[\gk(M)\simeq\gk(\cptwo)\times\prod_{i=1}^{d-1}\Omega^{2} G.\] 
   \end{letterlist} 
\end{theorem} 
\vspace{-1cm}~$\qqed$\bigskip 

Theorem~\ref{Mgaugedecomp} implies that to count the number of 
distinct homotopy types of the gauge groups $\gk(M)$ it suffices to 
do so in the special cases $\gk(S^{4})$ and $\gk(\cptwo)$. 
Further, at odd primes there is another decomposition, proved in~\cite{Th1}. 

\begin{theorem} 
   \label{oddpdecomp} 
   Let $G$ be a simply-connected, simple compact Lie group. Localize 
   rationally or at an odd prime $p$. Then there is a homotopy equivalence 
   \[\gk(\cptwo)\simeq\gk(S^{4})\times\Omega^{2} G.\] 
\end{theorem} 
\vspace{-1cm}~$\qqed$\bigskip 

Theorem~\ref{oddpdecomp} implies that the only possible difference between the 
number of distinct homotopy types of the gauge groups $\gk(S^{4})$ and 
$\gk(\cptwo)$ occurs at the prime $2$. The homotopy types $\gk(S^{4})$ for 
$G=Sp(2)$ were determined in~\cite{Th2}. 

\begin{theorem} 
   \label{Sp2gauge} 
   Let $G=Sp(2)$. The following hold: 
   \begin{letterlist} 
      \item if $\gk(S^{4})\simeq\mathcal{G}_{\ell}(S^{4})$ then $(40,k)=(40,\ell)$; 
      \item if $(40,k)=(40,\ell)$ then $\gk(S^{4})\simeq\mathcal{G}_{\ell}(S^{4})$ 
                when localized rationally or at any prime. 
   \end{letterlist} 
\end{theorem} 
\vspace{-1cm}~$\qqed$\bigskip 

Collectively, Theorems~\ref{Mgaugedecomp}, \ref{oddpdecomp} and~\ref{Sp2gauge} 
imply that the only case we need to consider to complete the proof of 
Theorem~\ref{count} is that of $\gk(\cptwo)$ at the prime $2$. Resolving 
this case occupies the remainder of the paper.

\section{A method for counting the homotopy types of gauge groups} 
\label{sec:countingmethod}

Return to the case of any simply-connected, simple compact Lie group $G$, 
let $M=S^{4}$ or~$\cptwo$ and let $\map_{k}(M,BG)$ be the component 
of the space of continuous unbased maps from $M$ to~$BG$ which contains 
the map inducing $P_{k}$. Similarly, let $\mapstar_{k}(M,BG)$ be the space 
of continuous pointed maps from $M$ to $BG$ which contains the map 
inducing $P_{k}$. Observe that there is a fibration 
\(\nameddright{\mapstar_{k}(M,BG)}{}{\map_{k}(M,BG)}{ev}{BG}\),  
where $ev$ evaluates a map at the basepoint of~$M$. Let $B\gk(M)$ be the 
classifying space of $\gk(M)$. By~\cite{AB,G}, 
there is a homotopy equivalence $B\gk(M)\simeq\map_{k}(M,BG)$. 
The evaluation fibration therefore determines a homotopy fibration sequence 
\begin{equation} 
  \label{evfib}    
  \lnamedddright{G}{\overline{\partial}_{k}}{\mapstar_{k}(M,BG)}{} 
     {B\gk(M)}{ev}{BG} 
\end{equation}   
which defines the map $\overline{\partial}_{k}$. In the case $M=S^{4}$, 
write $\partial_{k}$ for $\overline{\partial}_{k}$. 

The evaluation fibration~(\ref{evfib}) satisfies a naturality property. The fact that 
the pinch map 
\(\namedright{\cptwo}{\pi}{S^{4}}\) 
to the top cell induces an isomorphism $[S^{4},BG]\cong [\cptwo,BG]$ implies 
that it induces a one-to-one correspondence between the components 
of $\map_{k}(S^{4},BG)$ and $\map_{k}(\cptwo,BG)$ and the components 
of $\mapstar_{k}(S^{4},BG)$ and $\mapstar_{k}(\cptwo,BG)$. Moreover, 
it is well known that there is a homotopy equivalence 
$\mapstar_{k}(S^{4},BG)\simeq\mbox{Map}^{\ast}_{0}(S^{4},BG)$ 
for every~$k$, and by~\cite{Su} there is also a compatible homotopy equivalence 
$\mapstar_{k}(\cptwo,BG)\simeq\mapstar_{0}(\cptwo,BG)$, Therefore, 
writing $\Omega^{3}_{0} G$ for $\mapstar_{0}(S^{4},BG)$, we obtain a homotopy 
fibration diagram 
\begin{equation} 
  \label{natdgrm} 
  \diagram 
     G\rto^-{\partial_{k}}\ddouble & \Omega^{3}_{0} G\rto\dto^{\pi^{\ast}}  
         & B\gk(S^{4})\rto^-{ev}\dto & BG\ddouble \\ 
     G\rto^-{\overline{\partial}_{k}} & \mapstar_{0}(\cptwo,BG)\rto 
         & B\gk(\cptwo)\rto^-{ev} & BG. 
  \enddiagram 
\end{equation}  

The key to understanding the homotopy types of $\gk(S^{4})$ and $\gk(\cptwo)$ 
is understanding the homotopy classes $\partial_{k}$ and $\overline{\partial}_{k}$. 
The adjoint of $\partial_{k}$ can be identified. Let  
\(i\colon\namedright{S^{3}}{}{G}\) 
be the inclusion of the bottom cell and let 
\(1\colon\namedright{G}{}{G}\) 
be the identity map. Lang~\cite{L} identified the homotopy class of the 
triple adjoint of $\partial_{k}$ as follows. 

\begin{lemma} 
   \label{Lang} 
   The adjoint of 
   \(\namedright{G}{\partial_{k}}{\Omega^{3}_{0} G}\) 
   is homotopic to the Samelson product 
   \(\lnamedright{\mbox{$S^{3}\wedge G$}}{\langle ki,1\rangle}{G}\). 
   Consequently, the linearity of the Samelson product implies that 
   $\partial_{k}\simeq k\circ\partial_{1}$.~$\qqed$ 
\end{lemma} 

The following proposition, proved in~\cite{Th3}, uses the order of $\partial_{1}$ 
to estimate the number of homotopy types of the gauge groups $\gk(\cptwo)$. 

\begin{proposition} 
   \label{cptwomethod} 
   Suppose that the map 
   \(\namedright{G}{\partial_{1}}{\Omega^{3}_{0} G}\) 
   has order $m$. If $(m,k)=(m,\ell)$ then $\gk(\cptwo)$ is homotopy equivalent to 
   $\g_{\ell}(\cptwo)$ when localized rationally or at any prime.~$\qqed$ 
\end{proposition} 

\begin{remark} 
\label{cptwomethodremark} 
In fact, the proof of Proposition~\ref{cptwomethod} in~\cite{Th3} shows a stronger 
result: if $(m,k)=(m,\ell)$ then $\gk(\cptwo)$ is fibre homotopy equivalent to 
$\g_{\ell}(\cptwo)$ when localized rationally or at any prime. 
\end{remark} 

With this we can prove Theorem~\ref{count}~(a). 

\begin{proof}[Proof of Theorem~\ref{count}~(a)] 
By Theorem~\ref{Mgaugedecomp}, it suffices to prove the statement 
for the special cases $M=S^{4}$ and $M=\cptwo$. The $M=S^{4}$ 
case is the statement of Theorem~\ref{Sp2gauge}. By~\cite{Th2} the map 
\(\namedright{Sp(2)}{\partial_{1}}{\Omega^{3}_{0} Sp(2)}\) 
has order $40$. Part~(a) now follows immediately from Proposition~\ref{cptwomethod} 
and Remark~\ref{cptwomethodremark}. 
\end{proof} 

While this proves the statement in Theorem~\ref{count}~(a) there is a more subtle 
question that needs to be addressed. The proof depended only on the order of 
\(\namedright{Sp(2)}{\partial_{1}}{\Omega^{3}_{0} Sp(2)}\) 
while not taking into account the composite 
\(\nameddright{Sp(2)}{\partial_{1}}{\Omega^{3}_{0} Sp(2)}{\pi^{\ast}}{\mapstar_{0}(\cptwo,BSp(2)}\). 
Since $\gk(\cptwo)$ is the homotopy fibre of $\overline{\partial}_{1}=\pi^{\ast}\circ\partial_{1}$, 
it is more natural to consider properties of this map. Notice, though, that it makes no sense 
to talk about the order of the map~$\overline{\partial}_{1}$ since $\mapstar_{0}(\cptwo,BG)$ 
need not be an $H$-space. Instead, the factorization of $\overline{\partial}_{1}$ 
through $\partial_{1}$ lets us consider the ``order" of $\overline{\partial}_{k}$, by which we 
mean the least integer $n$ such that the composite 
\(\namedddright{G}{\partial_{1}}{\Omega^{3}_{0} G}{n}{\Omega^{3}_{0} G}{} 
      {\mapstar_{0}(\cptwo,BG)}\) 
is null homotopic. Potentially, it is possible that the ``order" of $\overline{\partial}_{1}$ 
is $20$ instead of $40$, in which case it is more reasonable to expect 
Theorem~\ref{count}~(a) having a g.c.d. condition involving $20$ rather than $40$. 

In Theorem~\ref{maporder} we show that the ``order" of $\overline{\partial}_{1}$ 
is $40$. This also plays an important role in proving Theorem~\ref{count}~(b). 
The build-up to Theorem~\ref{maporder} requires several calculations which 
occupy the next two sections.

\section{Preliminary information for the ``order" of $\overline{\partial}_{1}$} 
\label{sec:prelim1} 

Recall that $\hlgy{Sp(2);\mathbb{Z}}\cong\Lambda(x_{3},x_{7})$ 
where $x_{i}$ has degree~$i$. So $Sp(2)$ can be given a $CW$-structure 
with three cells, one in each of the dimensions $3$, $7$ and $10$. Let 
\(i\colon\namedright{S^{3}}{}{Sp(2)}\) 
be the inclusion of the bottom cell. Let $A$ be the $7$-skeleton of $Sp(2)$, 
so $A$ has two cells, one each in dimensions $3$ and~$7$. Let 
\(j\colon\namedright{A}{}{Sp(2)}\) 
be the skeletal inclusion. There are homotopy cofibrations 
\[\nameddright{S^{6}}{f_{1}}{S^{3}}{}{A}\] 
\[\nameddright{S^{9}}{f_{2}}{A}{}{Sp(2)}\] 
for some maps $f_{1}$ and $f_{2}$. Following Toda's notation~\cite{To} 
for the homotopy groups of spheres, it is well known that $f_{1}=\nu'$, 
where $\nu'$ represents a generator of $\pi_{6}(S^{3})\cong\mathbb{Z}/12\mathbb{Z}$ 
(see~\cite{MT}, for example). Let $\nu_{4}$ represent  
the integral generator of $\pi_{7}(S^{4})\cong\mathbb{Z}\oplus\mathbb{Z}/12\mathbb{Z}$, 
and by definition let $\nu_{m}=\Sigma^{m-4}\nu_{4}$ for $m\geq 5$. In 
particular, $\nu_{4}^{2}=\nu_{4}\circ\nu_{7}$. James~\cite{J} showed 
that $\Sigma f_{2}$ factors as the composite 
\(\nameddright{S^{10}}{2\nu_{4}^{2}}{S^{4}}{}{\Sigma A}$. 

\begin{lemma} 
   \label{Sp2cofib} 
   There is a homotopy cofibration sequence 
   \[\lnamedddright{S^{3}}{i}{Sp(2)}{}{S^{7}\vee S^{10}}{\Sigma\nu'-g}{S^{4}}\] 
   where $g=2\nu_{4}^{2}$. 
\end{lemma} 

\begin{proof} 
Define the space $C$ and the map $\delta$ by the homotopy cofibration sequence 
\begin{equation} 
  \label{Ccofib} 
  \namedddright{S^{3}}{i}{Sp(2)}{}{C}{\delta}{S^{4}}. 
\end{equation}  
In~\cite{Th2} it was shown that $C$ is homotopy equivalent to $S^{7}\vee S^{10}$; 
to also identify the homotopy class of $\delta$ we give an alternative description  
of the homotopy type of $C$. 

Let $q$ be the 
composite 
\(\nameddright{A}{j}{Sp(2)}{}{S^{7}}\) 
and notice that $q$ is the pinch map to the top cell. From this composite 
we obtain a homotopy pushout diagram 
\begin{equation} 
  \label{fpo} 
  \diagram 
      A\rto^-{j}\ddouble & Sp(2)\rto\dto & S^{10}\rto^-{\Sigma f_{2}}\dto^{g} 
              & \Sigma A\ddouble \\ 
      A\rto^-{q} & S^{7}\rto^-{\Sigma\nu'} & S^{4}\rto & \Sigma A 
  \enddiagram 
\end{equation} 
for some map $g$. James' factorization of $\Sigma f_{2}$ as 
\(\nameddright{S^{10}}{2\nu_{4}^{2}}{S^{4}}{}{\Sigma A}\) 
therefore implies that the difference $2\nu_{4}^{2}-g$ lifts to 
the homotopy fibre $F$ of 
\(\namedright{S^{4}}{}{\Sigma A}\). 
For dimensional reasons, this lift factors through the $10$-skeleton 
of $F$, which by a Serre spectral sequence calculation, is $S^{7}$. The 
composite 
\(S^{7}\hookrightarrow\namedright{F}{}{S^{4}}\) 
is homotopic to $\Sigma\nu'$, so $2\nu_{4}^{2}-g$ factors as the composite 
\(\nameddright{S^{10}}{h}{S^{7}}{\Sigma\nu'}{S^{4}}\) 
for some map $h$. By~\cite{To}, $\pi_{10}(S^{7})\cong\mathbb{Z}/12\mathbb{Z}$, 
generated by $\nu_{7}$, so $h=s\cdot\nu_{7}$ for some $s\in\mathbb{Z}/12\mathbb{Z}$, 
implying that $2\nu_{4}^{2}-g\simeq\Sigma\nu'\circ (s\cdot\nu_{7})$. 
By~\cite{To}, the Hopf invariant 
\(H\colon\namedright{\pi_{10}(S^{4})}{}{\pi_{10}(S^{7})}\) 
is an injection. But as $\Sigma\nu'\circ\nu_{7}\simeq\Sigma(\nu'\circ\nu_{6})$, 
we obtain $H(\Sigma\nu'\circ\nu_{7})=0$. Thus $\Sigma\nu'\circ\nu_{7}$ is null 
homotopic, and hence $g\simeq 2\nu_{4}^{2}$. 

The homotopy cofibration sequence 
\(\namedddright{S^{9}}{f_{2}}{A}{}{Sp(2)}{}{S^{10}}\) 
has a coaction 
\(\psi\colon\namedright{Sp(2)}{}{Sp(2)\vee S^{10}}\). 
Composing with the standard map 
\(\namedright{Sp(2)}{}{S^{7}}\) 
we obtain a composite 
\(\phi\colon\nameddright{Sp(2)}{\psi}{Sp(2)\vee S^{10}}{}{S^{7}\vee S^{10}}\). 
This coaction lets us use a Mayer-Vietorus like argument to produce a homotopy 
cofibration 
\(\nameddright{Sp(2)}{\phi}{S^{7}\vee S^{10}}{\Sigma\nu'-g}{S^{4}}\) 
from the pushout in the middle square of~(\ref{fpo}). 
Finally, notice that as $\phi\circ i$ is null homotopic, $\phi$ extends to a map  
\(\epsilon\colon\namedright{C}{}{S^{7}\vee S^{10}}\). 
As $\phi_{\ast}$ is an epimorphism, so is $\epsilon_{\ast}$. Comparing 
Poincar\'{e} series, $\epsilon_{\ast}$ must therefore be an isomorphism 
and so $\epsilon$ is a homotopy equivalence. Hence, up to homotopy, 
there is a cofibration sequence 
\(\namedddright{S^{3}}{i}{Sp(2)}{\phi}{S^{7}\vee S^{10}}{\Sigma\nu'-g}{S^{4}}\) 
where $g=2\nu_{4}^{2}$. 
\end{proof} 

By~\cite{Th2}, there is a homotopy commutative diagram 
\[\diagram 
      Sp(2)\rto^-{\partial_{1}}\dto & \Omega^{3}_{0} Sp(2) \\ 
      S^{7}\vee S^{10}\urto_-{a+b} & 
  \enddiagram\] 
where $a$ represents a generator of 
$\pi_{7}(\Omega^{3}_{0} Sp(2))\cong\pi_{10}(Sp(2))\cong\mathbb{Z}/8\mathbb{Z}$, 
and $b$ represents some class in 
$\pi_{10}(\Omega^{3}_{0} Sp(2))\cong\pi_{13}(Sp(2))\cong 
    \mathbb{Z}/4\mathbb{Z}\oplus\mathbb{Z}/2\mathbb{Z}$. 
Since $b$ has order at most $4$, we immediately obtain the following. 

\begin{lemma} 
   \label{4partial1} 
   There is a homotopy commutative diagram 
   \[\diagram 
          Sp(2)\rrto^-{\partial_{1}}\dto & & \Omega^{3}_{0} Sp(2)\dto^{4} \\ 
          S^{7}\vee S^{10}\rto^-{p_{1}} & S^{7}\rto^-{4a} & \Omega^{3}_{0} Sp(2) 
     \enddiagram\]
   where $p_{1}$ is the pinch map onto the left wedge summand.~$\qqed$ 
\end{lemma}

\section{Determining the ``order" of $\overline{\partial}_{1}$} 

Start with the homotopy cofibration sequence 
\begin{equation} 
  \label{bigcofib} 
  \namedddright{S^{3}}{\eta}{S^{2}}{}{\cptwo}{}{S^{4}} 
        \stackrel{\Sigma\eta}{\longrightarrow} S^{3} 
\end{equation}  
and the homotopy fibration 
\begin{equation} 
  \label{bigfib} 
  \nameddright{S^{3}}{}{Sp(2)}{}{S^{7}}. 
\end{equation}  
Applying the pointed mapping space functor $\mapstar(\ \ ,\ )$ with 
the spaces in the left variable coming from~(\ref{bigcofib}) and the spaces in 
the right variable coming from~(\ref{bigfib}), and writing $\mapstar(S^{t},X)$ 
as $\Omega^{t} X$, we obtain a homotopy fibration diagram 
\begin{equation} 
  \label{bigfibdgrm} 
  \diagram 
       \Omega^{3} S^{3}\rto^-{(\Sigma\eta)^{\ast}}\dto 
           & \Omega^{4} S^{3}\rto\dto & \mapstar(\cptwo,S^{3})\rto\dto 
           & \Omega^{2} S^{3}\rto^-{\eta^{\ast}}\dto & \Omega^{3} S^{3}\dto \\ 
       \Omega^{3} Sp(2)\rto^-{(\Sigma\eta)^{\ast}}\dto  
           & \Omega^{4} Sp(2)\rto\dto & \mapstar(\cptwo,Sp(2))\rto\dto   
           & \Omega^{2} Sp(2)\rto^-{\eta^{\ast}}\dto & \Omega^{3} Sp(2)\dto \\    
       \Omega^{3} S^{7}\rto^-{(\Sigma\eta)^{\ast}} 
           & \Omega^{4} S^{7}\rto & \mapstar(\cptwo,S^{7})\rto 
           & \Omega^{2} S^{7}\rto^-{\eta^{\ast}} & \Omega^{3} S^{7}.
  \enddiagram 
\end{equation}     

For the remainder of the section, localize all spaces and maps at $2$. 

\begin{lemma} 
   \label{pi3} 
   There is an isomorphism $\pi_{3}(\mapstar(\cptwo,Sp(2)))\cong\mathbb{Z}$ and 
   the map 
   \(\namedright{\Omega^{4} Sp(2)}{}{\mapstar(\cptwo,Sp(2))}\) 
   induces 
   \(\namedright{\mathbb{Z}}{2}{\mathbb{Z}}\) 
   on $\pi_{3}$. 
\end{lemma} 

\begin{proof} 
Apply $\pi_{3}$ to~(\ref{bigfibdgrm}), using the homotopy group information 
for spheres in~\cite{To} and $Sp(2)$ in~\cite{MT}, to obtain a diagram of exact sequences 
\begin{equation} 
  \label{pi3dgrm} 
  \diagram 
     \zmodfour\rto^-{(\Sigma\eta)^{\ast}}\dto & \zmodtwo\rto^-{a}\dto^{c} 
          & \pi_{3}(\mapstar(\cptwo,S^{3})\rto^-{b}\dto 
          & \zmodtwo\rto^-{\eta^{\ast}}\dto & \zmodfour\dto \\ 
     0\rto\dto & \mathbb{Z}\rto^-{d}\dto^{e} & \pi_{3}(\mapstar(\cptwo,Sp(2)))\rto\dto^{f} 
          & \zmodtwo\rto\dto & 0\dto \\ 
     0\rto & \mathbb{Z}\rto^-{g} & \pi_{3}(\mapstar(\cptwo,S^{7}))\rto & 0\rto & 0  
  \enddiagram 
\end{equation}  
where the maps $a$ through $g$ are to be determined. 

First, consider the map $a$. By~\cite{To}, $\pi_{3}(\Omega^{3} S^{3})\cong\zmodfour$ and 
$\pi_{3}(\Omega^{4} S^{3})\cong\zmodtwo$ are generated by the adjoints 
of $\nu'$ and $\nu'\circ\eta$ respectively. As $(\Sigma\eta)^{\ast}$ precomposes 
with $\eta$, this map is onto. Hence $a=0$. Second, consider the map $b$. By~\cite{To}, 
$\pi_{3}(\Omega^{2} S^{3})\cong\zmodtwo$ is generated by the adjoint of $\eta^{2}$, 
so~$\eta^{\ast}$ sends this to the adjoint of $\eta^{3}$. But this equals $2\nu'$, 
which generates $\pi_{3}(\Omega^{3} S^{3})$, so the map $b$ is an injection. 
Hence exactness along the top row in~(\ref{pi3dgrm}) implies that  
$\pi_{3}(\mapstar(\cptwo,S^{3})\cong 0$. This then implies that the map $f$ 
is an injection.  

Third, exactness along the bottom row in~(\ref{pi3dgrm}) implies that $g$ is 
an isomorphism. Fourth, by~\cite{MT}, the map 
\(\namedright{\pi_{3}(\Omega^{4} Sp(2))}{}{\pi_{3}(\Omega^{4} S^{7})}\) 
is 
\(\namedright{\mathbb{Z}}{4}{\mathbb{Z}}\), 
so in~(\ref{pi3dgrm}) we have $c=0$ and $e=4$. 

Finally, consider the middle row in~(\ref{pi3dgrm}). By exactness, either 
$\pi_{3}(\mapstar(\cptwo,Sp(2)))$ is isomorphic to $\mathbb{Z}$ with $d=2$ 
or it is isomorphic to $\mathbb{Z}\oplus\zmodtwo$ with $d$ being the inclusion 
of the integral summand. But as $f$ is an injection and 
$\pi_{3}(\mapstar(\cptwo,S^{7})\cong\mathbb{Z}$, it must be the first possibility 
that holds. Hence $\pi_{3}(\mapstar(\cptwo,Sp(2)))\cong\mathbb{Z}$ and the map 
\(\namedright{\Omega^{4} Sp(2)}{}{\mapstar(\cptwo,Sp(2))}\) 
induces 
\(\namedright{\mathbb{Z}}{2}{\mathbb{Z}}\) 
on~$\pi_{3}$. 
\end{proof} 

\begin{lemma} 
   \label{pi6} 
   There is an isomorphism $\pi_{6}(\mapstar(\cptwo,Sp(2)))\cong\zmodeight$,
   and the map 
   \(\namedright{\Omega^{4} Sp(2)}{}{\mapstar(\cptwo,Sp(2))}\) 
   induces an isomorphism on $\pi_{6}$. 
\end{lemma} 

\begin{proof} 
By~\cite{MT}, $\pi_{6}(\Omega^{3} Sp(2))\cong 0$ and 
$\pi_{6}(\Omega^{2} Sp(2))\cong 0$, so when $\pi_{6}$ is applied to the 
homotopy fibration in the middle row of~(\ref{bigfibdgrm}) we obtain an 
isomorphism 
\(\namedright{\pi_{6}(\Omega^{4} Sp(2))}{}{\pi_{6}(\mapstar(\cptwo,Sp(2)))}\). 
By~\cite{MT}, $\pi_{6}(\Omega^{4} Sp(2))\cong\zmodeight$. 
\end{proof} 

\begin{lemma} 
   \label{pi9} 
   The following hold: 
   \begin{letterlist} 
      \item there is an isomorphism $\pi_{9}(\mapstar(\cptwo,Sp(2)))\cong\zmodfour$;  
      \item the map 
               \(\namedright{\Omega^{4} Sp(2)}{}{\mapstar(\cptwo,Sp(2))}\) 
              sends the order~$4$ generator of 
              $\pi_{9}(\Omega^{4} Sp(2))\cong\zmodfour\oplus\zmodtwo$ to an 
              order~$4$ generator of $\pi_{9}(\mapstar(\cptwo,Sp(2)))$.  
   \end{letterlist} 
\end{lemma} 

\begin{proof} 
Apply $\pi_{9}$ to the top two rows of~(\ref{bigfibdgrm}) and use the homotopy 
group information for spheres in~\cite{To} and $Sp(2)$ in~\cite{MT} to obtain a 
diagram of exact sequences 
\begin{equation} 
  \label{pi9dgrm} 
  \spreaddiagramcolumns{-0.6pc} 
  \diagram 
      \zmodtwo\oplus\zmodtwo\rto^-{a}\dto^{e} 
          & \zmodfour\oplus\zmodtwo\rto^-{b}\dto^{f} 
          & \pi_{9}(\mapstar(\cptwo,S^{3}))\rto^-{c}\dto^{g} & \zmodtwo\rto^-{d}\dto^{h} 
          & \zmodtwo\oplus\zmodtwo\dto^{e} \\ 
      \zmodtwo\oplus\zmodtwo\rto^-{i} & \zmodfour\oplus\zmodtwo\rto^-{j} 
          & \pi_{9}(\mapstar(\cptwo,Sp(2)))\rto^-{k} & \zmodtwo\rto^-{\ell} 
          &  \zmodtwo\oplus\zmodtwo  
  \enddiagram 
\end{equation} 
where the maps $a$ through $\ell$ are to be determined. Note the two maps 
labelled $e$ are the same, that is, they are both 
\(\namedright{\pi_{9}(\Omega^{3} S^{3})}{}{\pi_{9}(\Omega^{3} Sp(2))}\). 

We aim to show that $j$ is onto, it is an injection on the $\zmodfour$ summand 
and it is trivial on the $\zmodtwo$ summand. These properties then imply that 
$\pi_{9}(\mapstar(\cptwo,Sp(2)))\cong\zmodfour$ and that the map 
\(\namedright{\Omega^{4} Sp(2)}{}{\mapstar(\cptwo,Sp(2))}\)  
sends the order~$4$ generator of 
$\pi_{9}(\Omega^{4} Sp(2))\cong\zmodfour\oplus\zmodtwo$ to an 
order~$4$ generator of $\pi_{9}(\mapstar(\cptwo,Sp(2)))$.  

First, by~\cite{MT}, the map 
\(\namedright{S^{3}}{}{Sp(2)}\) 
induces an isomorphism on $\pi_{11}$ and $\pi_{12}$, implying that $h$ 
and~$e$ respectively are isomorphisms. In fact, the generators of 
$\pi_{11}(Sp(2))$ and $\pi_{12}(Sp(2))$ are chosen as the images of those 
from $\pi_{11}(S^{3})$ and $\pi_{12}(S^{3})$ respectively, so $e$ and $h$ 
may be taken to be identity maps. As well, Theorem~5.1 and relation (5.1) 
of~\cite{MT} imply that the map $f$ is $2\oplus id$, where $id$ is the identity map. 

Consider the map $a$. Explicitly, $a$ is the map 
\(\namedright{\pi_{9}(\Omega^{3} S^{3})}{\eta^{\ast}}{\pi_{9}(\Omega^{4} S^{3})}\). 
Taking adjoints, $a$ sends a map 
\(\namedright{S^{12}}{t}{S^{3}}\) 
to the composite 
\(\nameddright{S^{13}}{\eta_{12}}{S^{11}}{t}{S^{3}}\). 
By~\cite{To}, the generators for $\pi_{12}(S^{3})\cong\zmodtwo\oplus\zmodtwo$ 
are $\eta_{3}\epsilon_{4}$ and $\mu_{3}$ and the generators for  
$\pi_{13}(S^{3})\cong\zmodfour\oplus\zmodtwo$ are~$\epsilon'$ 
and $\eta_{3}\mu_{4}$, of orders $4$ and $2$ respectively. Thus 
$a(\eta_{3}\epsilon_{4})=\eta_{3}\epsilon_{4}\eta_{12}$ and 
$a(\mu_{3})=\mu_{3}\eta_{12}$, neither of which are immediately generators 
of $\pi_{13}(S^{3})$. In general, two maps 
\(u\colon\namedright{S^{m}}{}{S^{n}}\) 
and 
\(v\colon\namedright{S^{m'}}{}{S^{n'}}\) 
satisfy $\Sigma^{n} v\circ\Sigma^{m'} u\simeq\pm\Sigma^{n'} u\circ\Sigma^{m} v$. 
In our case, we obtain $\epsilon_{5}\eta_{13}\simeq\eta_{5}\epsilon_{6}$ 
and $\mu_{5}\eta_{14}\simeq\eta_{5}\mu_{6}$, where the signs have disappeared 
since all composites are of order~$2$. Thus after two suspensions the image of $a$ 
consists of $\eta_{5}\epsilon_{6}\eta_{14}=\eta^{2}_{5}\epsilon_{7}$ and 
$\mu_{5}\eta_{14}=\eta_{5}\mu_{6}$. By~\cite[Lemma 6.6]{To}, 
$\eta^{2}_{5}\epsilon_{7}=2\Sigma^{2}\epsilon'$ and $\Sigma^{2}\epsilon'$ 
has order~$4$. Moreover, by~\cite[Theorem 7.3]{To}, the elements 
$\eta_{5}\mu_{6}$ and $\Sigma^{2}\epsilon'$ represent distinct generators 
in $\pi_{15}(S^{5})$. Thus after two suspensions $a$ induces an injection, 
implying that $a$ itself must be an injection.  

Let $x,y$ be generators for each of the 
summands in $\zmodfour\oplus\zmodtwo$. Since $a$ is an injection 
and consists of elements of order~$2$, the image is generated by 
$2sx+ty$ and $2s'x+t'y$ for $s,s',t,t'\in\zmodtwo$. The injectivity of $a$ implies 
that at least one of $t$ or $t'$ is equal to $1$. As $f=2\oplus id$, the image 
of $f\circ a$ is $ty$ and $t'y$. The commutativity of the leftmost square 
in~(\ref{pi9dgrm}) then implies that the image of $i$ is $ty$ and $ty'$. Therefore 
the image of $i$ in $\mathbb{Z}/4\mathbb{Z}\oplus\mathbb{Z}/2\mathbb{Z}$ 
is trivial when projected to $\mathbb{Z}/4\mathbb{Z}$ and onto when projected 
to $\mathbb{Z}/2\mathbb{Z}$. Thus, exactness 
along the bottom row of~(\ref{pi9dgrm}) implies that $j$ 
is an injection on the $\zmodfour$ summand and is trivial on the 
$\zmodtwo$ summand. 

Consider the map $\ell$. Since $h$ and $e$ are identity maps, it is equivalent 
to determine $d$. Explicitly, $d$ is the map 
\(\namedright{\pi_{9}(\Omega^{2} S^{3})}{(\Sigma\eta)^{\ast}}{\pi_{9}(\Omega^{3} S^{3})}\). 
Taking adjoints, $d$ sends a map 
\(\namedright{S^{11}}{s}{S^{3}}\) 
to the composite 
\(\nameddright{S^{12}}{\eta_{11}}{S^{11}}{s}{S^{3}}\). 
By~\cite{To}, $\pi_{11}(S^{3})\cong\zmodtwo$ is generated 
by $\epsilon_{3}$, and as above, $\pi_{12}(S^{3})\cong\zmodtwo\oplus\zmodtwo$ is 
generated by $\eta_{3}\epsilon_{4}$ and $\mu_{3}$. Observe that 
$d(\epsilon_{3})=\epsilon_{3}\eta_{11}$, so as before, after 
two suspensions the image of $d$ is $\epsilon_{5}\eta_{13}=\eta_{5}\epsilon_{6}$, 
which is nontrivial. Therefore after two suspensions $d$ induces an injection, 
implying that $d$ itself is an injection. Hence $\ell$ is an injection, implying 
that $k$ is the zero map and $j$ is onto. 
\end{proof} 

In what follows we work with $\mapstar(\cptwo,BSp(2))$, and note that the 
pointed exponential law implies that 
$\Omega\mapstar(\cptwo,BSp(2))\simeq\mapstar(\cptwo,Sp(2))$. By 
Lemma~\ref{pi3}, $\pi_{4}(\mapstar(\cptwo,BSp(2)))\cong\mathbb{Z}$. Let 
\(\gamma\colon\namedright{S^{4}}{}{\mapstar(\cptwo,BSp(2))}\) 
represent a generator. The second statement in Lemma~\ref{pi3} implies 
that there is a homotopy commutative diagram 
\begin{equation} 
  \label{thetadgrm} 
  \diagram 
      S^{4}\dto^{\underline{2}}\rto^-{\theta} & \Omega^{3} Sp(2)\dto \\ 
      S^{4}\rto^-{\gamma} & \mapstar(\cptwo,BSp(2)) 
  \enddiagram 
\end{equation}  
where $\theta$ represents a generator of $\pi_{4}(\Omega^{3} Sp(2))\cong\mathbb{Z}$. 

\begin{lemma} 
   \label{gammanu} 
   The composite 
   \(\nameddright{S^{7}}{\nu_{4}}{S^{4}}{\gamma}{\mapstar(\cptwo,BSp(2))}\) 
   has order~$8$. 
\end{lemma} 

\begin{proof} 
Consider the diagram 
\[\diagram 
      S^{7}\rto^-{\nu_{4}}\dto^{\underline{4}} & S^{4}\dto^{\underline{2}}\rto^-{\theta} 
          & \Omega^{3} Sp(2)\dto \\ 
      S^{7}\rto^-{\nu_{4}} & S^{4}\rto^-{\gamma} & \mapstar(\cptwo,BSp(2)).  
  \enddiagram\] 
The left square homotopy commutes since $\nu_{4}$ has Hopf invariant 
one. The right square homotopy commutes by~(\ref{thetadgrm}). We claim 
that the upper direction around this diagram is a map of order~$2$. This 
would imply that $\gamma\circ\nu_{4}\circ\underline{4}$ has order~$2$, 
implying in turn that $\gamma\circ\nu_{4}$ has order~$8$. 

It remains to show that the upper direction around the diagram has order~$2$. 
By Lemma~\ref{pi6}, the map 
\(\namedright{\Omega^{3} Sp(2)}{}{\mapstar(\cptwo,BSp(2))}\) 
induces an isomorphism on $\pi_{7}$, implying that it suffices to show that 
the map $\theta\circ\nu_{4}$ has order~$2$. If  
\(c\colon\namedright{S^{7}}{}{Sp(2)}\) 
is the adjoint of $\theta$, it is equivalent to show that $c\circ\nu_{7}$ 
has order~$2$. Consider the composite 
\[\namedddright{S^{10}}{\nu_{7}}{S^{7}}{c}{Sp(2)}{q}{S^{7}}.\] 
By~\cite{MT}, $q$ induces an isomorphism on $\pi_{10}$, so it 
suffices to show that $q\circ c\circ\nu_{7}$ has order~$2$. By~\cite{To}, 
$\pi_{10}(S^{7})\cong\zmodeight$ is generated by $\nu_{7}$, and by~\cite{MT}, 
$q\circ c\simeq\pm 4$. Thus $q\circ c\circ\nu_{7}\simeq\pm 4\nu_{7}$ has 
order~$2$. 
\end{proof} 

Recall that the map 
\(\namedright{S^{10}}{g}{S^{4}}\) 
in Lemma~\ref{Sp2cofib} is $2\nu_{4}^{2}$. 

\begin{proposition} 
   \label{gammag} 
   The composite 
   \(\nameddright{S^{10}}{g}{S^{4}}{\gamma}{\mapcptwo}\) 
   has order~$2$.~$\qqed$  
\end{proposition} 

\begin{proof} 
Consider the diagram 
\begin{equation} 
  \label{gammagdgrm} 
  \diagram 
     S^{10}\rto^-{2\nu_{7}} & S^{7}\rto^-{\nu_{4}}\dto^{\lambda} 
         & S^{4}\dto^{\gamma} \\ 
     & \Omega^{3} Sp(2)\rto^-{j} & \mapstar(\cptwo,BSp(2)) 
  \enddiagram 
\end{equation} 
where $j$ is the usual map and $\lambda$ will be defined momentarily. 
By Lemma~\ref{pi6}, $\pi_{7}(\mapstar(\cptwo,BSp(2)))\cong\zmodeight$ 
so Lemma~\ref{gammanu} implies that $\gamma\circ\nu_{4}$ represents 
a generator. Lemma~\ref{pi6} also states that the map 
\(\namedright{\Omega^{3} Sp(2)}{}{\mapstar(\cptwo,BSp(2))}\) 
induces an isomorphism on $\pi_{7}$. Thus $\gamma\circ\nu_{4}$ 
lifts to a map 
\(\lambda\colon\namedright{S^{7}}{}{\Omega^{3} Sp(2)}\) 
which represents a generator of $\pi_{7}(\Omega^{3} Sp(2))\cong\zmodeight$ 
and makes~(\ref{gammagdgrm}) homotopy commute. 

Mimura and Toda chose a generator of $\pi_{10}(Sp(2))\cong\zmodeight$ 
they called $[\nu_{7}]$. Let 
\(c\colon\namedright{S^{10}}{}{Sp(2)}\) 
be the triple adjoint of $\lambda$, so $c$ is also a generator 
of $\pi_{10}(Sp(2))$. Then $c\simeq u\cdot[\nu_{7}]$, where $u\in\zmodeight$ 
is a unit. In~\cite{MT} it is also shown that $[\nu_{7}]\nu_{10}$ represents the order~$4$ 
generator in $\pi_{13}(Sp(2))\cong\zmodfour\oplus\zmodtwo$. Thus 
$c\circ\nu_{10}$ also has order~$4$. Taking adjoints, $\lambda\circ\nu_{7}$ 
has order~$4$. Now reconsider~(\ref{gammagdgrm}). By Lemma~\ref{pi9}, 
the map 
\(\namedright{\Omega^{3} Sp(2)}{}{\mapstar(\cptwo,BSp(2))}\) 
sends the order~$4$ generator of 
$\pi_{10}(\Omega^{3} Sp(2))\cong\zmodfour\oplus\zmodtwo$ to an 
order~$4$ generator of $\pi_{10}(\mapstar(\cptwo,BSp(2)))$. Thus 
$j\circ\lambda\circ\nu_{7}$ has order~$4$, implying that $j\circ\lambda\circ 2\nu_{7}$ 
has order~$2$. The homotopy commutativity of~(\ref{gammagdgrm}) therefore 
implies that $\gamma\circ\nu_{4}\circ 2\nu_{7}$ has order~$2$. That is, 
$\gamma\circ g$ has order~$2$. 
\end{proof} 

We will also need to know how 
\(\namedright{S^{7}}{\Sigma\nu'}{S^{4}}\) 
composes with $\gamma$. 

\begin{lemma} 
   \label{gammanuprime} 
   The composite 
   \(\nameddright{S^{7}}{\Sigma\nu'}{S^{4}}{\gamma}{\mapstar(\cptwo,BSp(2))}\) 
   has order at most~$2$. 
\end{lemma} 

\begin{proof} 
First observe that the homotopy fibration 
\(\nameddright{\Omega^{3} Sp(2)}{}{\mapstar(\cptwo,BSp(2))}{}{\Omega Sp(2)}\) 
implies that there is an exact sequence 
\(\nameddright{\pi_{5}(\Omega^{3} Sp(2))}{}{\pi_{5}(\mapstar(\cptwo,BSp(2)))} 
     {}{\pi_{5}(\Omega Sp(2))}\). 
By~\cite{MT}, $\pi_{5}(\Omega^{3} Sp(2))\cong\pi_{5}(\Omega Sp(2))\cong 0$, 
implying that $\pi_{5}(\mapstar(\cptwo,BSp(2)))\cong 0$. 

Next, by~\cite{To}, $2\nu'=\eta_{3}^{3}$. Therefore the composite 
\(\namedddright{S^{7}}{\underline{2}}{S^{7}}{\Sigma\nu'}{S^{4}}{\gamma 
       }{\mapstar(\cptwo,BSp(2))}\) 
factors through the composite 
\(\nameddright{S^{5}}{\eta_{4}}{S^{4}}{\gamma}{\mapstar(\cptwo,BSp(2))}\). 
The latter is null homotopic since $\pi_{5}(\mapstar(\cptwo,BSp(2)))\cong 0$, 
and hence $\gamma\circ\Sigma\nu'\circ\underline{2}$ is null homotopic. 
Consequently, $\gamma\circ\Sigma\nu'$ has order at most~$2$. 
\end{proof} 

\begin{theorem} 
   \label{maporder} 
   The composite 
   \(\namedddright{Sp(2)}{\partial_{1}}{\Omega^{3}_{0} Sp(2)}{4}{\Omega^{3}_{0} Sp(2)} 
          {}{\mapstar_{0}(\cptwo,BSp(2))}\) 
   is nontrivial.  
\end{theorem} 

\begin{proof} 
By Lemma~\ref{4partial1} there is a homotopy commutative diagram 
\[\diagram 
       Sp(2)\rrto^-{\partial_{1}}\dto & & \Omega^{3}_{0} Sp(2)\dto^{4} \\ 
       S^{7}\vee S^{10}\rto^-{p_{1}} & S^{7}\rto^-{4a} & \Omega^{3}_{0} Sp(2)  
  \enddiagram\] 
and by Lemma~\ref{Sp2cofib} there is a homotopy cofibration 
\(\lnameddright{Sp(2)}{}{S^{7}\vee S^{10}}{\Sigma\nu'-g}{S^{4}}\). 
Aiming for a contradiction, suppose that the composite 
\(\namedddright{Sp(2)}{\partial_{1}}{\Omega^{3}_{0} Sp(2)}{4}{\Omega^{3}_{0} Sp(2)} 
          {}{\mapstar_{0}(\cptwo,BSp(2))}\) 
is null homotopic. Then there is an extension 
\begin{equation} 
  \label{extdgrm} 
  \diagram 
        S^{7}\vee S^{10}\rto^-{p_{1}}\dto^{\Sigma\nu'-g} & S^{7}\rto^-{4a} 
            & \Omega^{3}_{0} Sp(2)\dto \\ 
        S^{4}\rrto^-{\theta} & & \mapstar_{0}(\cptwo,BSp(2)) 
  \enddiagram 
\end{equation}  
for some map $\theta$. 

We claim that $\theta$ must represent a generator of 
$\pi_{4}(\mapstar_{0}(\cptwo,BSp(2)))\cong\mathbb{Z}$. 
Restrict~(\ref{extdgrm}) to the~$S^{7}$ summand of $S^{7}\vee S^{10}$. 
Since $a$ represents a generator of $\pi_{7}(\Omega^{3} Sp(2))\cong\zmodeight$, 
the map $4a$ has order~$2$. By Lemma~\ref{pi6} the map 
\(\namedright{\Omega^{3}_{0} Sp(2)}{}{\mapstar_{0}(\cptwo,BSp(2))}\) 
induces an isomorphism on $\pi_{7}$. Thus the upper direction 
around~(\ref{extdgrm}) has order~$2$ when restricted to $S^{7}$. 
Therefore the homotopy commutativity of the diagram implies that 
$\theta\circ\Sigma\nu'$ has order~$2$. If $\theta$ did not represent 
a generator of $\pi_{4}(\mapstar_{0}(\cptwo,BSp(2)))\cong\mathbb{Z}$, 
then as we are localized at $2$, we must have $\theta\simeq 2t\cdot\gamma$ 
for some integer~$t$. But then 
$\theta\circ\Sigma\nu'\simeq (t\gamma)\circ\underline{2}\circ\Sigma\nu'\simeq 
    (t\gamma)\circ\Sigma\nu'\circ\underline{2}$, 
where the right homotopy holds since~$\Sigma\nu'$ is a suspension. 
By Lemma~\ref{gammanuprime}, $\gamma\circ\Sigma\nu'\circ\underline{2}$ 
is null homotopic, implying that $\theta\circ\Sigma\nu'$ is null homotopic, 
a contradiction. Hence $\theta$ represents a generator of 
$\pi_{4}(\mapstar_{0}(\cptwo,BSp(2))\cong\mathbb{Z}$, so $\theta\simeq\pm\gamma$. 

Now restrict~(\ref{extdgrm}) to the $S^{10}$ summand of $S^{7}\vee S^{10}$. 
The upper direction around the diagram is null homotopic, implying that 
$\theta\circ g\simeq\pm\gamma\circ g$ is null homotopic. But this contradicts 
Proposition~\ref{gammag}. Hence the composite 
\(\namedddright{Sp(2)}{\partial_{1}}{\Omega^{3}_{0} Sp(2)}{4}{\Omega^{3}_{0} Sp(2)} 
          {}{\mapstar_{0}(\cptwo,BSp(2))}\) 
must be nontrivial. 
\end{proof}

\section{Preliminary information for Theorem~\ref{count}~(b)} 
\label{sec:prelim2} 

Now we turn to Theorem~\ref{count}~(b). A key ingredient is Theorem~\ref{Acount}, 
which states that if there is a homotopy equivalence $\gk(\cptwo)\simeq\g_{\ell}(\cptwo)$ 
(rather than a fibre homotopy equivalence) then $(20,k)=(20,\ell)$. The build-up 
to Theorem~\ref{Acount} takes most of the next two sections. 

In this section we establish $2$-primary 
properties of certain homotopy sets related to~$Sp(2)$ and $\gk(\cptwo)$. 
Localize all spaces and maps at $2$. Three sequences of spaces will be used. 
The first is the homotopy fibration sequence 
\begin{equation} 
  \label{Sp2fib} 
  \namedddright{\Omega S^{7}}{\delta}{S^{3}}{j}{Sp(2)}{}{S^{7}}. 
\end{equation} 
The second is the homotopy cofibration sequence 
\begin{equation} 
  \label{Acofib} 
  \nameddddright{S^{6}}{\nu'}{S^{3}}{i}{A}{q}{S^{7}}{\Sigma\nu'}{S^{4}}   
\end{equation} 
where $\nu'$ is the attaching map for the top cell of $A$, $i$ is the 
inclusion of the bottom cell and $q$ is the pinch map to the top cell.
Toda's notation is used for $\nu'$; it represents a generator of  
$\pi_{6}(S^{3})\cong\mathbb{Z}/4\mathbb{Z}$. 
Applying $\mapstar(\ \ ,BSp(2))$ to the homotopy cofibration sequence 
\[\nameddddright{S^{3}}{\eta}{S^{2}}{}{\cptwo}{}{S^{4}}{\Sigma\eta}{S^{3}}.\] 
gives the third sequence of spaces, the homotopy fibration sequence 
\begin{equation} 
   \label{etafib} 
   \nameddddright{\Omega^{2} Sp(2)}{(\Sigma\eta)^{\ast}}{\Omega^{3} Sp(2)} 
      {}{\mapstar(\cptwo,BSp(2))}{}{\Omega Sp(2)}{\eta^{\ast}}{\Omega^{2} Sp(2)}. 
\end{equation} 

To simplify notation, let $N=\mapstar(\cptwo,BSp(2))$. Applying $[A,\ \ ]$ 
to the homotopy fibration sequence~(\ref{etafib}) gives an exact sequence 
\begin{equation} 
  \label{Aexact} 
  \nameddddright{[A,\Omega^{2} Sp(2)]}{(\Sigma\eta)^{\ast}}{[A,\Omega^{3} Sp(2)]}{} 
      {[A,N]}{}{[A,\Omega Sp(2)]}{\eta^{\ast}}{[A,\Omega^{2} Sp(2)]}. 
\end{equation}  

The goal of this section is to calculate $[A,N]$. The starting point 
is the following result of Choi, Hirato and Mimura~\cite{CHM}. 

\begin{lemma} 
   \label{CHMlemma} 
   There is an isomorphism $[A,\Omega^{3} Sp(2)]\cong\mathbb{Z}/8\mathbb{Z}$ 
   and a representative of the generator is the composite 
   \(\nameddright{A}{s}{Sp(2)}{\partial_{1}}{\Omega^{3}_{0} Sp(2)}\). 
   $\qqed$ 
\end{lemma} 

Given Lemma~\ref{CHMlemma}, to calculate $[A,N]$ in~(\ref{Aexact}) 
we determine the image of $(\Sigma\eta)^{\ast}$ and the kernel of~$\eta^{\ast}$. 
To do so we first collect some information on the homotopy groups of $S^{3}$, $S^{7}$ 
and $Sp(2)$. For $m\geq 3$, let 
\(\eta_{m}\colon\namedright{S^{m+1}}{}{S^{m}}\) 
be $\Sigma^{m-2}\eta$. 

\begin{lemma}[Toda~\cite{To}]  
   \label{Toda} 
   In the relevant dimensions the groups $\pi_{i}(S^{3})$ are:  
   \medskip\newline\medskip\hspace{1cm}  
   \begin{tabular}{|l|c|c|c|c|c|} 
            \hline 
            $i$ & $5$ & $6$ & $8$ & \hspace{2mm} $9$\hspace{2mm} 
                 & \hspace{2mm}$10$\hspace{2mm}\ \\ \hline
            \mbox{$2$-component} & $\mathbb{Z}/2\mathbb{Z}$
                 & $\mathbb{Z}/4\mathbb{Z}$ & $\mathbb{Z}/2\mathbb{Z}$ 
                 & $0$ & $0$ \\ \hline
            \mbox{generator} & $\eta_{3}^{2}$ & $\nu'$  
                 & $\nu^{\prime}\eta_{6}^{2}$ & -- & -- \\ \hline 
   \end{tabular} \newline 
   Further, by~\cite[Lemma 5.7]{To} the composite 
   \(\nameddright{S^{8}}{\Sigma^{2}\nu'}{S^{5}}{\eta_{3}^{2}}{S^{3}}\) 
   is null homotopic. 
 
   In the relevant dimensions the groups $\pi_{i}(S^{7})$ are: 
   \smallskip\newline\smallskip\hspace{1cm} 
   \begin{tabular}{|l|c|c|c|} 
      \hline 
      $i$ & $7$ & $10$ & \hspace{2mm}$11$\hspace{2mm} \\ \hline
       \mbox{$2$-component} & $\mathbb{Z}$ & $\mathbb{Z}/8\mathbb{Z}$ 
              & $0$ \\ \hline
       \mbox{generator} & $\iota_{7}$ & $\nu_{7}$ & -- \\ \hline 
  \end{tabular}  \medskip\newline  
   Further, $2\cdot\nu'_{7}\simeq\Sigma^{4}\nu'$.~$\qqed$  
\end{lemma} 

\begin{lemma}[Mimura-Toda~\cite{MT}]  
   \label{MT} 
   In the relevant dimensions, the groups $\pi_{i}(Sp(2))$ are: 
   \medskip\newline\medskip\hspace{1cm} 
   \begin{tabular}{|l|c|c|c|c|c|c|c|} 
         \hline 
         $i$ & $4$ & $5$ & \hspace{1mm} $6$\hspace{1mm} & \hspace{2mm} $7$\hspace{2mm}  
               & \hspace{2mm} $8$\hspace{2mm} 
               & \hspace{2mm} $9$\hspace{2mm} & $13$ \\ \hline
         \mbox{$2$-component} & $\mathbb{Z}/2\mathbb{Z}$
               & $\mathbb{Z}/2\mathbb{Z}$ & $0$  
               & $\mathbb{Z}$ & $0$ & $0$ 
               & $\mathbb{Z}/4\mathbb{Z}\oplus\mathbb{Z}/2\mathbb{Z}$ \\ \hline
     \end{tabular} \newline 
    Further, the map 
    \(\namedright{S^{5}}{\eta_{4}}{S^{4}}\) 
    induces an isomorphism 
    \(\namedright{\pi_{4}(Sp(2))}{\cong}{\pi_{5}(Sp(2))}\).~$\qqed$ 
\end{lemma}  

We first aim towards Proposition~\ref{etastar}, which describes the 
map $\eta^{\ast}$ in~(\ref{Aexact}). 

\begin{lemma} 
   \label{eta1} 
   The map 
   \(\namedright{S^{3}}{i}{A}\) 
   induces an isomorphism 
   \(\namedright{[A,\Omega Sp(2)]}{i^{\ast}}{[S^{3},\Omega Sp(2)]}\). 
\end{lemma} 

\begin{proof} 
The homotopy cofibration~(\ref{Acofib}) induces an exact sequence of groups 
\[\namedddright{[S^{7},\Omega Sp(2)]}{q^{\ast}}{[A,\Omega Sp(2)]}{i^{\ast}} 
       {[S^{3},\Omega Sp(2)]}{(\nu')^{\ast}}{[S^{6},\Omega Sp(2)]}.\] 
Identifying the first, third and fourth terms using Lemma~\ref{MT} we obtain 
an exact sequence 
\[\namedddright{0}{q^{\ast}}{[A,\Omega Sp(2)]}{i^{\ast}} 
       {\mathbb{Z}/2\mathbb{Z}}{(\nu')^{\ast}}{\mathbb{Z}}.\] 
Since $\mathbb{Z}$ is torsion-free, $(\nu')^{\ast}$ must be 
the zero map. The lemma now follows. 
\end{proof} 

\begin{lemma} 
   \label{eta2} 
   The map 
   \(\namedright{S^{3}}{i}{A}\) 
   induces an isomorphism 
   \(\namedright{[A,\Omega^{2} Sp(2)]}{i^{\ast}}{[S^{3},\Omega^{2} Sp(2)]}\). 
\end{lemma} 

\begin{proof} 
The homotopy cofibration~(\ref{Acofib}) induces an exact sequence of groups 
\[\namedddright{[S^{7},\Omega^{2} Sp(2)]}{q^{\ast}}{[A,\Omega^{2} Sp(2)]}{i^{\ast}} 
       {[S^{3},\Omega^{2} Sp(2)]}{(\nu')^{\ast}}{[S^{6},\Omega^{2} Sp(2)]}.\] 
Identifying the first, third and fourth terms using Lemma~\ref{MT} we obtain 
an exact sequence 
\[\namedddright{0}{q^{\ast}}{[A,\Omega^{2} Sp(2)]}{i^{\ast}}{\mathbb{Z}/2\mathbb{Z}} 
         {(\nu')^{\ast}}{0}.\] 
Therefore $i^{\ast}$ is an isomorphism. 
\end{proof} 

\begin{proposition} 
   \label{etastar} 
   The map 
   \(\namedright{\Omega Sp(2)}{\eta^{\ast}}{\Omega^{2} Sp(2)}\) 
   induces an isomorphism 
   \(\namedright{[A,\Omega Sp(2)]}{\eta^{\ast}}{[A,\Omega^{2} Sp(2)]}\). 
\end{proposition} 

\begin{proof} 
Consider the commutative diagram 
\begin{equation} 
   \label{etastardgrm} 
   \diagram 
        [A,\Omega Sp(2)]\rto^-{a}\dto^{b} 
            & [S^{3},\Omega Sp(2)]\dto^{d} \\ 
        [A,\Omega^{2} Sp(2)]\rto^-{c} & [S^{3},\Omega^{2} Sp(2)]. 
   \enddiagram 
\end{equation} 
Here, $a$ and $c$ are induced by $i^{\ast}$ while $b$ and $c$ 
are induced by $\eta^{\ast}$. By Lemma~\ref{eta1}, $a$ is 
an isomorphism; by Lemma~\ref{eta2}, $c$ is an isomorphism. 
Observe that $d$ is the map 
\(\namedright{\pi_{4}(Sp(2))}{}{\pi_{5}(Sp(2))}\) 
induced by precomposing with 
\(\namedright{S^{5}}{\eta_{4}}{S^{4}}\), 
which is an isomorphism by Lemma~\ref{MT}. Thus $d$ is an isomorphism. 
The commutativity of~(\ref{etastardgrm}) therefore implies that $b$ is 
an isomorphism as well, proving the lemma. 
\end{proof} 

Next, we aim towards Proposition~\ref{Setastar}, which describes 
the map $(\Sigma\eta)^{\ast}$ in~(\ref{Aexact}). 

\begin{lemma} 
   \label{Seta1} 
   The map 
   \(\namedright{S^{3}}{i}{A}\) 
   induces an isomorphism 
   \(\namedright{[A,\Omega^{2} S^{3}]}{i^{\ast}} 
         {[S^{3},\Omega^{2} S^{3}]\cong\mathbb{Z}/2\mathbb{Z}}\). 
\end{lemma} 

\begin{proof} 
The homotopy cofibration~(\ref{Acofib}) induces an exact sequence of groups 
\[\namedddright{[S^{7},\Omega^{2} S^{3}]}{q^{\ast}}{[A,\Omega^{2} S^{3}]}{i^{\ast}} 
       {[S^{3},\Omega^{2} S^{3}]}{(\nu')^{\ast}}{[S^{6},\Omega^{2} S^{3}]}.\] 
Identifying the first, third and fourth terms using Lemma~\ref{Toda} we obtain 
an exact sequence 
\[\namedddright{0}{q^{\ast}}{[A,\Omega^{2} S^{3}]}{i^{\ast}} 
       {\mathbb{Z}/2\mathbb{Z}}{(\nu')^{\ast}}{\mathbb{Z}/2\mathbb{Z}}.\] 
The generator of $\pi_{5}(S^{3})$ is $\eta_{3}^{2}$, so 
$(\nu')^{\ast}(\eta_{3}^{2})$ is the composite 
\(\nameddright{S^{8}}{\Sigma^{2}\nu'}{S^{5}}{\eta_{3}^{2}}{S^{3}}\), 
which is null homotopic by Lemma~\ref{Toda}. Thus $(\nu')^{\ast}$ is the 
zero map. The lemma now follows. 
\end{proof} 

\begin{lemma} 
   \label{Seta2} 
   The map 
   \(\namedright{S^{3}}{i}{A}\) 
   induces multiplication by $4$ in the map   
   \(\namedright{\mathbb{Z}\cong [A,\Omega^{4} S^{7}]}{i^{\ast}} 
         {[S^{3},\Omega^{4} S^{7}]\cong\mathbb{Z}}\).  
\end{lemma} 

\begin{proof} 
The homotopy cofibration~(\ref{Acofib}) induces an exact sequence of groups 
\[\namedddright{[S^{7},\Omega^{4} S^{7}]}{q^{\ast}}{[A,\Omega^{4} S^{7}]}{i^{\ast}} 
       {[S^{3},\Omega^{4} S^{7}]}{(\nu')^{\ast}}{[S^{6},\Omega^{4} S^{7}]}.\] 
Identifying the first, third and fourth terms using Lemma~\ref{Toda} we obtain 
an exact sequence 
\[\namedddright{0}{q^{\ast}}{[A,\Omega^{4} S^{7}]}{i^{\ast}}{\mathbb{Z}} 
         {(\nu')^{\ast}}{\mathbb{Z}/8\mathbb{Z}}.\] 
Thus $[A,\Omega^{4} S^{7}]$ is isomorphic to the kernel of $(\nu')^{\ast}$. 
To determine this kernel, observe that the generator of 
$[S^{3},\Omega^{4} S^{7}]\cong\pi_{7}(S^{7})\cong\mathbb{Z}$ is $\iota_{7}$ 
while the generator of $[S^{6},\Omega^{4} S^{7}]\cong\mathbb{Z}/8\mathbb{Z}$ 
is $\nu_{7}$. By Lemma~\ref{Toda}, $2\nu_{7}\simeq\Sigma^{4}\nu'$ so we obtain 
$(\Sigma\nu')^{\ast}(\iota_{7})=\iota_{7}\circ\Sigma^{4}\nu'\simeq 
        \Sigma^{4}\nu'\simeq 2\nu_{7}$. 
Therefore the kernel of~$(\nu')^{\ast}$ is isomorphic to $\mathbb{Z}$ and 
is generated by $4\iota_{7}$. 
\end{proof} 

\begin{lemma} 
   \label{Seta3} 
   The map 
   \(\namedright{S^{3}}{j}{Sp(2)}\) 
   induces an isomorphism 
   \(\namedright{[A,\Omega^{2} S^{3}]}{j_{\ast}}{[A,\Omega^{2} Sp(2)]}\). 
\end{lemma} 

\begin{proof} 
The maps 
\(\namedright{S^{3}}{i}{A}\) 
and 
\(\namedright{S^{3}}{j}{Sp(2)}\) 
induce a commutative diagram 
\begin{equation} 
  \label{Seta3dgrm} 
  \diagram 
      [A,\Omega^{2} S^{3}]\rto^-{a}\dto^{b} & [A,\Omega^{2} Sp(2)]\dto^{d} \\ 
      [S^{3},\Omega^{2} S^{3}]\rto^-{c} & [S^{3},\Omega^{2} Sp(2)]. 
  \enddiagram 
\end{equation} 
By Lemma~\ref{Seta1}, the map $b$ is an isomorphism. By~\cite{MT}, $j$ 
induces an isomorphism 
\(\namedright{\pi_{5}(S^{3})}{}{\pi_{5}(Sp(2))}\), 
so the map $c$ is an isomorphism. By Lemma~\ref{eta2}, $d$ is an 
isomorphism. Hence in~(\ref{Seta3dgrm}) the maps $b$, $c$ and $d$ are 
isomorphisms, implying that $a$ is as well. 
\end{proof} 

\begin{lemma} 
   \label{Seta4} 
   The fibration connecting map 
   \(\namedright{\Omega S^{7}}{\delta}{S^{3}}\) 
   induces the zero map on the homotopy sets  
   \(\namedright{[A,\Omega^{4} S^{7}]}{\delta_{\ast}}{[A,\Omega^{3} S^{3}]}\). 
\end{lemma} 

\begin{proof} 
By Lemma~\ref{Seta2}, $[A,\Omega^{4} S^{7}]\cong\mathbb{Z}$ and  
\(\namedright{\mathbb{Z}\cong [A,\Omega^{4} S^{7}]}{i^{\ast}} 
         {[S^{3},\Omega^{4} S^{7}]\cong\mathbb{Z}}\) 
is multiplication by $4$. Therefore, if $\epsilon$ is a generator 
of $[A,\Omega^{4} S^{7}]$, then there is a homotopy commutative diagram 
\begin{equation} 
  \label{relate2dgrm} 
  \diagram 
        S^{3}\rto^-{E^{4}}\dto^{i} & \Omega^{4} S^{7}\rto^-{4} & \Omega^{4} S^{7} \\ 
        A\urrto_{\epsilon} & & 
  \enddiagram 
\end{equation}  
where $E^{4}$ is the four-fold suspension. Consider $\delta_{\ast}(\epsilon)$, 
which is the composite 
\(\nameddright{A}{\epsilon}{\Omega^{4} S^{7}}{\Omega^{3}\delta}{\Omega^{3} S^{3}}\). 
By~(\ref{relate2dgrm}), $\Omega^{3}\delta\circ\epsilon\circ i$ is homotopic to 
the composite 
\(\namedddright{S^{3}}{E^{4}}{\Omega^{4} S^{7}}{4}{\Omega^{4} S^{7}} 
        {\Omega^{3}\delta}{\Omega^{3} S^{3}}\). 
But the latter composite is null homotopic since 
$\pi_{6}(S^{3})\cong\mathbb{Z}/4\mathbb{Z}$. Thus 
$\Omega^{3}\delta\circ\epsilon$ extends through the cofibre of $i$ to a map 
\(\namedright{S^{7}}{}{\Omega^{3} S^{3}}\). 
But by Lemma~\ref{Toda}, $\pi_{10}(S^{3})\cong 0$. Thus 
$\Omega^{3}\delta\circ\epsilon$ is null homotopic. That is, 
$\delta_{\ast}(\epsilon)=0$. As $\epsilon$ generates $[A,\Omega^{4} S^{7}]$, 
this implies that $\delta_{\ast}$ is the zero map. 
\end{proof} 
   
\begin{corollary} 
   \label{Seta5} 
   The map 
   \(\namedright{S^{3}}{j}{Sp(2)}\) 
   induces a monomorphism 
   \(\namedright{[A,\Omega^{3} S^{3}]}{j_{\ast}}{[A,\Omega^{3} Sp(2)]}\). 
\end{corollary} 

\begin{proof} 
The homotopy fibration~(\ref{Sp2fib}) induces an exact sequence 
\(\nameddright{[A,\Omega^{4} S^{7}]}{\delta_{\ast}}{[A,\Omega^{3} S^{3}]} 
       {j_{\ast}}{[A,\Omega^{3} Sp(2)]}\). 
By Lemma~\ref{Seta4}, $\delta_{\ast}$ is the zero map. Hence $j_{\ast}$ 
is a monomorphism. 
\end{proof} 

\begin{lemma} 
   \label{Seta6} 
   The map 
   \(\namedright{\Omega^{2} S^{3}}{(\Sigma\eta)^{\ast}}{\Omega^{3} S^{3}}\) 
   induces an injection 
   \(\namedright{[A,\Omega^{2} S^{3}]}{(\Sigma\eta)^{\ast}}{[A,\Omega^{3} S^{3}]}\). 
\end{lemma} 

\begin{proof} 
By Lemma~\ref{Seta3}, $[A,\Omega^{2} S^{3}]\cong\mathbb{Z}/2\mathbb{Z}$ 
and there is an isomorphism 
\(\namedright{[A,\Omega^{2} S^{3}]}{i^{\ast}}{[S^{3},\Omega^{2} S^{3}]}\). 
Therefore, if $\epsilon$ is a generator of $[A,\Omega^{2} S^{3}]$, then 
there is a homotopy commutative diagram 
\begin{equation} 
  \label{relate3dgrm} 
  \diagram 
      S^{3}\rto^-{\overline{\eta}_{3}^{2}}\dto^{i} & \Omega^{2} S^{3} \\ 
      A\urto_{\epsilon} & 
  \enddiagram 
\end{equation} 
where $\overline{\eta}_{3}^{2}$ is the adjoint of $\eta_{3}^{2}$. Consider 
the composite 
\(\nameddright{A}{\epsilon}{\Omega^{2} S^{3}}{(\Sigma\eta)^{\ast}}{\Omega^{3} S^{3}}\). 
By~(\ref{relate3dgrm}), $(\Sigma\eta)^{\ast}\circ\epsilon\circ i$ is homotopic to 
the composite 
\(\nameddright{S^{3}}{\overline{\eta}_{3}^{2}}{\Omega^{2} S^{3}}{(\Sigma\eta)^{\ast}} 
       {\Omega^{3} S^{3}}\), 
which in turn is homotopic to the adjoint of $\eta_{3}^{3}$. Therefore, as $\eta_{3}^{3}$ 
is nontrivial so is $(\Sigma\eta)^{\ast}\circ\epsilon$. Hence  
\(\namedright{[A,\Omega^{2} S^{3}]}{}{[A,\Omega^{3} S^{3}]}\) 
is an injection because it is nonzero on the generator of 
$[A,\Omega^{2} S^{3}]\cong\mathbb{Z}/2\mathbb{Z}$.  
\end{proof} 

\begin{lemma} 
   \label{Seta7} 
   The map 
   \(\namedright{\Omega^{2} S^{3}}{(\Sigma\eta)^{\ast}}{\Omega^{3} S^{3}}\) 
   induces an injection 
   \(\namedright{[A,\Omega^{2} Sp(2)]}{(\Sigma\eta)^{\ast}}{[A,\Omega^{3} Sp(2)]}\). 
\end{lemma} 

\begin{proof} 
The map 
\(\namedright{\Omega^{2} X}{(\Sigma\eta)^{\ast}}{\Omega^{3} X}\) 
is natural with respect to maps 
\(\namedright{X}{}{Y}\), 
so applying this to the case  
\(\namedright{BS^{3}}{Bj}{BSp(2)}\) 
and taking homotopy sets $[A,\hspace{2mm}]$ gives a commutative diagram 
\begin{equation} 
   \label{Seta7dgrm} 
   \diagram 
        [A,\Omega^{2} S^{3}]\rto^-{a}\dto^{b} 
            & [A,\Omega^{3} S^{3}]\dto^{d} \\ 
        [A,\Omega^{2} Sp(2)]\rto^-{c} & [A,\Omega^{3} Sp(2)]. 
   \enddiagram 
\end{equation} 
Here, $a$ and $c$ are induced by $(\Sigma\eta)^{\ast}$ while $b$ and $c$ 
are induced by $j_{\ast}$. By Lemma~\ref{Seta3}, $b$ is an isomorphism 
and, by Corollary~\ref{Seta5}, $d$ is an injection. By Lemma~\ref{Seta6}, 
$a$ is an injection. Therefore the commutativity of~(\ref{Seta7dgrm}) implies 
that $c$ is also an injection. 
\end{proof} 

Now we identify the injection in Lemma~\ref{Seta7}. 

\begin{proposition} 
   \label{Setastar} 
   The injection 
   \(\namedright{[A,\Omega^{2} Sp(2)]}{(\Sigma\eta)^{\ast}}{[A,\Omega^{3} Sp(2)]}\) 
   in Lemma~\ref{Seta7} is the injection 
   \(\namedright{\mathbb{Z}/2\mathbb{Z}}{}{\mathbb{Z}/8\mathbb{Z}}\). 
\end{proposition} 

\begin{proof} 
By Lemmas~\ref{Seta1} and~\ref{Seta3}, 
$[A,\Omega^{2} Sp(2)]\cong\mathbb{Z}/2\mathbb{Z}$,  
and by Lemma~\ref{CHMlemma}, $[A,\Omega^{3} Sp(2)]\cong\mathbb{Z}/8\mathbb{Z}$.  
Since 
\(\namedright{[A,\Omega^{2} Sp(2)]}{}{[A,\Omega^{3} Sp(2)]}\) 
is an injection, the lemma follows. 
\end{proof} 

Finally, we calculate $[A,N]$. 

\begin{proposition} 
   \label{AN} 
   There is an isomorphism of sets $[A,N]\cong\mathbb{Z}/4\mathbb{Z}$ 
   and the map 
   \(\namedright{\Omega^{3}_{0} Sp(2)}{}{N}\) 
   induces an epimorphism 
   \(\namedright{\mathbb{Z}/8\mathbb{Z}\cong [A,\Omega^{3} Sp(2)]}{} 
           {[A,N]\cong\mathbb{Z}/4\mathbb{Z}}\). 
\end{proposition} 

\begin{proof} 
By Propositions~\ref{etastar} and~\ref{Setastar} the exact sequence in~(\ref{Aexact}) 
simplifies to an exact sequence 
\[\nameddddright{0}{}{\mathbb{Z}/2\mathbb{Z}}{} 
       {\mathbb{Z}/8\mathbb{Z}\cong [A,\Omega^{3} Sp(2)]}{}{[A,N]}{}{0}.\] 
The proposition follows immediately. 
\end{proof}

\section{The proof of Theorem~\ref{count}~(b)} 
\label{sec:counting} 

In this section we prove Theorem~\ref{count}~(b). A key ingredient is  
Theorem~\ref{Acount},which To simplify notation, 
let $N_{k}=\mapstar_{k}(\cptwo,BSp(2))$.  Applying $[A,\ \ ]$ 
to~(\ref{natdgrm}) in the $G=Sp(2)$ case we obtain a commutative diagram 
\begin{equation} 
  \label{lowerbddgrm} 
  \diagram 
      [A,Sp(2)]\rto^-{(\partial_{k})_{\ast}}\ddouble 
           & [A,\Omega^{3}_{0} Sp(2)]\rto\dto  
           & [A,B\gk(S^{4})]\rto\dto & [A,BSp(2)]\ddouble \\ 
      [A,Sp(2)]\rto^-{(\overline{\partial}_{k})_{\ast}} 
           & [A,N_{k}]\rto & [A,B\gk(\cptwo)]\rto & [A,BSp(2)]. 
  \enddiagram 
\end{equation} 
Since $A$ is connected, there is an isomorphism 
$[A,N_{0}]\cong [A,N]$, where $N=\mapstar(\cptwo,BSp(2))$ and 
$[A,N]$ was calculated in Proposition~\ref{AN}. By~\cite{Su}, $N_{k}\simeq N_{0}$ 
for every $k\in\mathbb{Z}$, so $[A,N_{k}]\cong [A,N]$ as well. 

\begin{lemma} 
   \label{partialimage} 
   The image of $(\overline{\partial}_{k})_{\ast}$ is $\mathbb{Z}/(4/(4,k))\,\mathbb{Z}$. 
\end{lemma} 

\begin{proof} 
By Lemma~\ref{CHMlemma}, $[A,\Omega^{3}_{0} Sp(2)]\cong\mathbb{Z}/8\mathbb{Z}$ 
and the composite 
\(\nameddright{A}{s}{Sp(2)}{\partial_{1}}{\Omega^{3}_{0} Sp(2)}\) 
represents a generator. By Lemma~\ref{Lang}, $\partial_{k}\simeq k\cdot\partial_{1}$, 
so the image of $(\partial_{k})^{\ast}$ in~(\ref{lowerbddgrm}) is 
$\mathbb{Z}/(8/(8,k))\,\mathbb{Z}$. By Proposition~\ref{AN}, the map 
\(\namedright{[A,\Omega^{3}_{0} Sp(2)]}{}{[A,N_{k}]}\) 
is reduction mod $4$. Thus the commutativity of the left square 
in~(\ref{lowerbddgrm}) implies that the image of $(\overline{\partial}_{k})^{\ast}$ 
is $\mathbb{Z}/(4/(4,k))\,\mathbb{Z}$. 
\end{proof} 

\begin{lemma} 
   \label{ABSp} 
   There is an isomorphism $[A,BSp(2)]\cong 0$. 
\end{lemma} 

\begin{proof} 
The dimension of $A$ is $7$ and the map 
\(\namedright{BSp(2)}{}{BSp(\infty)}\) 
induced by the standard inclusion of $Sp(2)$ into $Sp(\infty)$ is 
$10$-connected. Therefore $[A,BSp(2)]\cong [A,BSp(\infty)]$. 
But $[A,BSp(\infty)]$ is $\widetilde{K}_{Sp}(A)$, the reduced symplectic 
$K$-theory of $A$.  Since $\widetilde{K}_{Sp}(S^{4m-1})=0$ for 
every~$m\geq 1$, applying $\widetilde{K}_{Sp}$ to the homotopy cofibration 
\(\nameddright{S^{3}}{}{A}{}{S^{7}}\) 
shows that $\widetilde{K}_{Sp}(A)\cong 0$. 
\end{proof} 

\begin{proposition} 
   \label{ABgk} 
   There is an isomorphism of sets $[A,B\gk(\cptwo)]\cong\mathbb{Z}/(4,k)\,\mathbb{Z}$. 
\end{proposition} 

\begin{proof} 
By Lemma~\ref{ABSp} the exact sequence in the bottom row of~(\ref{lowerbddgrm}) is 
\[\namedddright{[A,Sp(2)]}{(\overline{\partial}_{k})^{\ast}}{[A,\Omega^{3}_{0} Sp(2)]} 
       {}{[A,B\gk(\cptwo)]}{}{0}.\] 
Thus $[A,B\gk(\cptwo)]$ is the cokernel of $(\overline{\partial}_{k})^{\ast}$. 
By Lemma~\ref{partialimage}, this cokernel is $\mathbb{Z}/(4,k)\,\mathbb{Z}$. 
\end{proof} 

What we would like to say is that a homotopy equivalence $\gk(\cptwo)\simeq\g_{\ell}(\cptwo)$ 
implies that there is a bijection of sets $[A,B\gk(\cptwo)]\cong [A,B\g_{\ell}(\cptwo)]$. This 
would imply that $(4,k)=(4,\ell)$ and prove Theorem~\ref{count}~(b). However, since $A$ 
is not a suspension, and possibly not even a co-$H$-space, the desired implication is 
not immediate through adjunction. Instead, a different argument is needed. 

\begin{lemma} 
   \label{pi3Bgauge} 
   For any $k\in\mathbb{Z}$ we have $\pi_{3}(B\gk(\cptwo))\cong 0$. 
\end{lemma} 

\begin{proof} 
There is a homotopy fibration 
\begin{equation} 
    \label{pi3eq1} 
    \nameddright{\mapstar_{k}(\cptwo,BSp(2))}{}{B\gk(\cptwo)}{}{BSp(2)} 
\end{equation} 
and the homotopy cofibration sequence 
\(\namedddright{S^{3}}{\eta}{S^{2}}{}{\cptwo}{}{S^{4}}\) 
induces a homotopy fibration sequence 
\begin{equation} 
   \label{pi3eq2} 
   \namedddright{\Omega^{3}_{0} Sp(2)}{}{\mapstar_{k}(\cptwo,BSp(2))}{} 
         {\Omega Sp(2)}{\eta^{\ast}}{\Omega^{2} Sp(2)}. 
\end{equation} 
Apply $\pi_{3}$ to~(\ref{pi3eq2}). By Lemma~\ref{MT}, 
$\pi_{3}(\Omega^{3}_{0} Sp(2))\cong\pi_{6}(Sp(2))\cong 0$, and 
$\pi_{3}(\eta^{\ast})$ is the same as the map 
\(\namedright{\pi_{4}(Sp(2))}{(\Sigma^{2}\eta)^{\ast}}{\pi_{5}(Sp(2))}\), 
which is an isomorphism. Hence $\pi_{3}(\mapstar_{k}(\cptwo,BSp(2)))\cong 0$. 
Now apply $\pi_{3}$ to~(\ref{pi3eq1}). By connectivity, $\pi_{3}(BSp(2))\cong 0$ 
so as $\pi_{3}(\mapstar_{k}(\cptwo,BSp(2)))\cong 0$ we obtain $\pi_{3}(B\gk(\cptwo))\cong 0$. 
\end{proof} 

Consider the homotopy cofibration sequence 
\(\namedddright{S^{3}}{}{A}{q}{S^{7}}{\Sigma\nu'}{S^{4}}\). 
Applying $[\ \ ,B\gk(\cptwo) ]$ gives an exact sequence of pointed sets 
\[\namedddright{[S^{4},B\gk(\cptwo)]}{(\Sigma\nu)^{\ast}}{[S^{7},B\gk(\cptwo)]}{q^{\ast}} 
        {[A,B\gk(\cptwo)]}{}{[S^{3},B\gk(\cptwo)]}.\]        
By Lemma~\ref{pi3Bgauge}, $\pi_{3}(B\gk(\cptwo))\cong 0$, so we really have 
an exact sequence of pointed sets
\begin{equation} 
   \label{Anuexact} 
   \namedddright{[S^{4},B\gk(\cptwo)]}{(\Sigma\nu')^{\ast}}{[S^{7},B\gk(\cptwo)]}{q^{\ast}} 
        {[A,B\gk(\cptwo)]}{}{0}. 
\end{equation} 
Note that $[S^{4},B\gk(\cptwo)]$ and $[S^{7},B\gk(\cptwo)]$ are groups and 
$(\Sigma\nu')^{\ast}$ is a group homomorphism, but $[A,B\gk(\cptwo)]$ is only 
a set since $A$ may not be a co-$H$-space, and $q^{\ast}$ is therefore only a 
map of sets. So exactness and the fact that $q^{\ast}$ is an epimorphism implies 
that there is a bijection between the set $[A,B\gk(\cptwo)]$ and the group 
$\mbox{coker}\,(\Sigma\nu')^{\ast}$. 

Now we prove Theorems~\ref{Acount} and~\ref{count}~(b), which are stated integrally. 

\begin{theorem} 
   \label{Acount} 
   If $\gk(\cptwo)\simeq\g_{\ell}(\cptwo)$ then $(20,k)=(20,\ell)$. 
\end{theorem} 

\begin{proof} 
By Theorems~\ref{oddpdecomp} and~\ref{Sp2gauge}~(a) it suffices to prove 
the $2$-components of the g.c.d. conditions. That is, it suffices to prove that 
if there is a $2$-local homotopy equivalence $\gk(\cptwo)\simeq\g_{\ell}(\cptwo)$ 
then $(4,k)=(4,\ell)$. 

Let 
\(e\colon\namedright{\gk(\cptwo)}{}{\g_{\ell}(\cptwo)}\) 
be a homotopy equivalence. Then for any path-connected space~$X$ 
there is an isomorphism $[\Sigma X,B\gk(\cptwo)]\cong [X,\gk(\cptwo)]$ 
which is natural for maps 
\(\namedright{\Sigma X}{\Sigma f}{\Sigma X'}\). 
Applying this to~(\ref{Anuexact}), together with the map $e$, gives a commutative diagram 
\begin{equation} 
  \label{Aquotient} 
  \diagram 
      [S^{4},B\gk(\cptwo)]\rto^-{(\Sigma\nu')^{\ast}}\dto^{\cong} 
           & [S^{7},B\gk(\cptwo)]\rto^-{q^{\ast}}\dto^{\cong} 
           & [A,B\gk(\cptwo)]\rto & 0 \\ 
      [S^{4},B\g_{\ell}(\cptwo)]\rto^-{(\Sigma\nu')^{\ast}} 
           & [S^{7},B\g_{\ell}(\cptwo)]\rto^-{q^{\ast}} & [A,B\g_{\ell}(\cptwo)]\rto & 0    
  \enddiagram 
\end{equation}  
where the top and bottom row are exact and the vertical isomorphisms are 
induced by adjunction and applying $e_{\ast}$. In the top row,  
$[A,B\gk(\cptwo)]$ bijects with $\mbox{coker}\,(\Sigma\nu')^{\ast}$, and in the bottom row,  
$[A,B\gk(\cptwo)]$ bijects with $\mbox{coker}\,(\Sigma\nu')^{\ast}$. Therefore there is a bijection 
between $[A,B\gk(\cptwo)]$ and $[A,B\g_{\ell}(\cptwo)]$. 

By Proposition~\ref{ABgk}, $[A,B\gk(\cptwo)]\cong\mathbb{Z}/(4,k)\,\mathbb{Z}$ 
and similarly $[A,B\g_{\ell}(\cptwo)]\cong\mathbb{Z}/(4,\ell)\,\mathbb{Z}$. Therefore 
the bijection between $[A,B\gk(\cptwo)]$ and $[A,B\g_{\ell}(\cptwo)]$ implies that $(4,k)=(4,\ell)$. 
\end{proof} 

\begin{proof}[Proof of Theorem~\ref{count}~(b)] 
By Theorem~\ref{Mgaugedecomp}, it suffices to prove the statement 
for the special cases $M=S^{4}$ and $M=\cptwo$. The $M=S^{4}$ 
case is the statement of Theorem~\ref{Sp2gauge}. Thus we are reduced to proving 
that if $\gk(\cptwo)$ is fibre homotopy equivalent to $\g_{\ell}(\cptwo)$ then $(40,k)=(40,\ell)$. 

Let $u_{k}$ and $v_{k}$ be the $2$-component of the integers $(20,k)$ and $(40,k)$ 
respectively. Observe that~$u_{k}$ can be $0, 1, 2$ or $4$. If $u_{k}<4$, then 
$v_{k}=u_{k}$ and so $(40,k)=(20,k)$; otherwise $v_{k}$ is either $4$ or $8$.

Since $\gk(\cptwo)$ is fibre homotopy equivalent to $\g_{\ell}(\cptwo)$ the two 
are also homotopy equivalent. Theorem~\ref{Acount} therefore implies that $(20, k)=(20, \ell)$. 
In particular, this implies that $u_{k}=u_{\ell}$. If $u_{k}<4$ then $(20,k)=(40,k)$ and 
similarly $(20,\ell)=(40,\ell)$. Therefore $(40,k)=(40,\ell)$, as asserted. If $u_{k}=4$ then 
$v_{k}$ is either $4$ or $8$; further $u_{k}=4$ implies that $u_{\ell}=4$ and so $v_{\ell}$ 
is either $4$ or $8$. If $v_{k}=v_{\ell}$ then we obtain $(40,k)=(40,\ell)$, as asserted. 
It remains, then, to deal with the case when $v_{k}=8$ and $u_{\ell}=4$ or vice-versa. 

Assume that $v_{k}=8$ and $v_{\ell}=4$ (the other case is similar). We will show 
that $\gk(\cptwo)$ cannot be fiber homotopy equivalent to $\g_{\ell}(\cptwo)$. Note 
that $k=8K$ and $\ell=4L$ where $K$ is some number and~$L$ is an odd number. 
Write $L=2L'+1$, giving $\ell=4L=8L'+4$. The definition of $\overline{\partial}_{k}$ 
as $\pi^{\ast}\circ\partial_{k}$ and Lemma~\ref{Lang} imply that 
\[\overline{\partial}_k =\pi^{\ast}\circ\partial_k\simeq\pi^{\ast}\circ 8K\circ\partial_1\qquad 
     \overline{\partial}_{\ell} =\pi^{\ast}\circ\partial_{\ell}\simeq\pi^{\ast}\circ (8L'+4)\circ\partial_1.\] 
Localize at $2$. By~\cite{Th2} the map $\partial_{1}$ has order~$8$. Therefore 
$\overline{\partial}_{k}$ is null homotopic and 
$\overline{\partial}_{\ell}\simeq\pi^{\ast}\circ 4\partial_{1}\simeq\pi^{\ast}\circ\partial_{4} 
    \simeq\overline{\partial}_{4}$. 
Let 
\(h\colon\namedright{\gk(\cptwo)}{}{\g_{\ell}(\cptwo)}\) 
be a fibre homotopy equivalence. Then there is a homotopy commutative diagram 
\[
  \diagram 
      \gk(\cptwo)\rto^-{r_k}\dto^{h} 
           & Sp(2)\rto^-{\overline{\partial}_k}\ddouble 
           & \mapstar(\cptwo,BSp(2))\\ 
      \g_{\ell}(\cptwo)\rto^-{r_{\ell}} 
           & Sp(2)\rto^-{\overline{\partial}_{\ell}} & \mapstar(\cptwo,BSp(2))
  \enddiagram 
\] 
where both rows are fibrations. 
Since $\overline{\partial}_k$ is null homotopic, the map $r_{k}$ has a right homotopy inverse 
$s:Sp(2)\to\gk(\cptwo)$. The homotopy commutativity of the square therefore implies that 
$r_{\ell}\circ h\circ s$ is a right homotopy inverse for $r_{\ell}$. This implies that 
$\overline{\partial}_{\ell}$ is null homotopic. However, $\overline{\partial}_{\ell}\simeq\overline{\partial}_{4}$ 
and Theorem~\ref{maporder} implies that $\overline{\partial}_{4}$ is nontrivial, a contradiction. 
Therefore the case $v_{k}=8$ and $v_{\ell}=4$ cannot occur. 

Hence in all cases, the fibre homotopy equivalence between $\gk(\cptwo)$ 
and $\g_{\ell}(\cptwo)$ implies that $(40,k)=(40,\ell)$. 
\end{proof}

\bibliographystyle{amsplain}

\end{document}